\documentclass[11pt]{amsart}
\usepackage{graphicx}
\usepackage[T1]{fontenc}
\usepackage{amsmath,amsfonts,amssymb,amscd,verbatim,amsthm,amsgen,xspace, graphicx,color, breqn}
\usepackage{srcltx}
\usepackage{enumerate}

\newcommand{\ka}{\mathfrak{k}}
\newcommand{\p}{\mathfrak{p}}
\newcommand{\g}{\mathfrak{g}}

\newcommand{\pf}{\begin{proof}}
\newcommand{\epf}{\end{proof}}
\newcommand{\eq}{\begin{equation}}
\newcommand{\eeq}{\end{equation}}
\newcommand{\eqn}{\begin{equation*}}
\newcommand{\eeqn}{\end{equation*}}
\newcommand{\twedge}{\textstyle\bigwedge}

\newcommand{\frg}{\mathfrak{g}}

\newcommand{\frk}{\mathfrak{k}}

\newcommand{\frp}{\mathfrak{p}}

\newcommand{\frsl}{\mathfrak{sl}}

\newcommand{\bbC}{\mathbb{C}}

\newcommand{\bbZ}{\mathbb{Z}}

\newcommand{\cspan }{\operatorname{span}}

\theoremstyle{plain}
\newtheorem{theorem}{Theorem}[section]
\newtheorem{cor}[theorem]{Corollary}
\newtheorem{prop}[theorem]{Proposition}
\newtheorem{lemma}[theorem]{Lemma}

\begin{document}

\title[The centralizer of $K$ in $U(\mathfrak{g}) \otimes C(\mathfrak{p})$ for the group $SO_e(4,1)$]{The centralizer of $K$ in $U(\mathfrak{g}) \otimes C(\mathfrak{p})$ for the group $SO_e(4,1)$}
\author[A. Prli\'c]{Ana Prli\' c*}\thanks{*The author was supported by grant no. 4176 of the Croatian Science Foundation and by the QuantiXLie Center of Excellence.}
\address{Department of Mathematics\\University of Zagreb\\ 10 000 Zagreb\\ Croatia}
\email{anaprlic@math.hr}



\begin{abstract}
Let $G$ be the Lie group $SO_e(4,1)$, with maximal compact subgroup $K = S(O(4) \times O(1))_e\cong SO(4)$. Let $\g=\mathfrak{so}(5,\mathbb{C})$ be the complexification of the Lie algebra $\g_0 = \mathfrak{so}(4,1)$ of $G$, and let $U(\g)$ be the universal enveloping algebra of $\g$. Let $\g = \ka \oplus \p$ be the Cartan decomposition of $\g$, and $C(\p)$ the  Clifford algebra of $\p$ with respect to the trace form $B(X, Y) = \text{tr}(XY)$ on $\p$. In this paper we give explicit generators of the algebra $(U(\g) \otimes C(\p))^{K}$.
\end{abstract}

\maketitle

\section{Introduction}
\label{intro}

Let $G$ be a connected real reductive Lie group with Cartan involution $\Theta$, such that $K = G^{\Theta}$ is a maximal compact subgroup of $G$. Let $\g$ be the complexification of the Lie algebra $\g_0$ of $G$, and let $U(\g)$ be the universal enveloping algebra of $\g$. Let 
$\g_0 = \ka_0 \oplus \p_0$ and $\g = \ka \oplus \p$ be the Cartan decompositions of $\g_0$ and $\g$ corresponding to $\Theta$. Let $B$ be a fixed nondegenerate invariant symmetric bilinear form on $\g_0$, such that $B$ is negative definite on $\ka_0$ and positive definite on $\p_0$, and such that $\ka_0$ is orthogonal to $\p_0$. We use the same letter $B$ for the extension of $B$ to $\g$. In general, one can define $B$ as a suitable extension of the Killing form, while for matrix examples one can take the trace form.
Let $C(\p)$ be the Clifford algebra of $\p$ with respect to $B$. 

If $X$ is an irreducible $(\g, K)$--module, then $X$ is completely determined by the action of the centralizer of $\ka$ in $U(\g)$ on any $K$--isotypic component of $X$. This was proved by Harish-Chandra in \cite{HC}. The problem with this approach to determination of all $X$ is that the structure of the algebra $U(\g)^{K}$ is very complicated to describe. 

There is a version of this theorem proved in \cite{PR} where one considers modules $X \otimes S$, where $X$ is a $(\g, K)$--module, and $S$ is the spin module for $C(\p)$. Let $\tilde{K}$  be the spin double cover of $K$. Then $X \otimes S$ is a $(U(\g) \otimes C(\p), \tilde{K})$--module and it is proved in \cite{PR} that an irreducible $(U(\g) \otimes C(\p), \tilde{K})$--module is determined by the action of the algebra $(U(\g) \otimes C(\p))^{K}$ on any nontrivial $\tilde{K}$--isotypic component of $X \otimes S$. 

The algebra $(U(\g) \otimes C(\p))^{K}$ is completely determined for the group $SU(2, 1)$ in \cite{Pr1}. Except for the obvious generators from the center $Z(\ka)$ of $U(\ka)$ and $C(\p)^{K}$, there are two more generators: the Dirac operator, and its $\ka$--version called $\ka$--Dirac. In this paper we prove that a similar result holds for the group $SO_e(4,1)$. In this case $\ka$ is equivalent to the product of two copies of $\mathfrak{sl}(2, \mathbb{C})$ and we prove that the centralizer of $\ka$ in $U(\g) \otimes C(\p)$ is generated by two copies of the Casimir operator of $\mathfrak{sl}(2, C)$, by one generator from  $C(\p)^K$, by the Dirac operator and by the $\ka$--Dirac.  

The algebra $(U(\g) \otimes C(\p))^{K}$ is also important for the Dirac induction developed in \cite{PR} and further studied in \cite{Pr2}. For the reader's convenience, let us recall the basic definitions and results important in the setting of the Dirac operator. 

The Dirac operator corresponding to a real reductive Lie group was defined by Parthasarathy \cite{P}, where it was used to construct the discrete series representations. The final form of the construction was given by Atiyah and Schmid \cite{AS}. The algebraic version of the Dirac operator and the notion of Dirac cohomology were defined by Vogan in \cite{V} and further studied by Huang and Pand\v zi\'c \cite{HP1}, \cite{HP2}. Let $b_i$ be a basis of $\p$ and let $d_i$ be the dual basis with respect to $B$. Then
\[
D = \sum_{i} b_i \otimes d_i\in U(\g)\otimes C(\p).
\]
The sum does not depend of the choice of basis $b_i$. If $X$ is a $(\g, K)$--module, and $S$ a spin module for $C(\p)$, then the Dirac cohomology of $X$ is
\[
H_{D}(X) = \text{Ker}D /\text{Im}D \cap \text{Ker}D.
\]
It is a $\tilde{K}$--module. In \cite{PR} the opposite construction was described, i.e. a construction of a $(\g, K)$--module $X$ whose Dirac cohomology contains a given irreducible $\tilde{K}$--module. It was proved in \cite{PR} and \cite{Pr2} that some discrete series representation can be constructed in that way. The final goal is to prove that the same construction works for all discrete series representation, and also for some other modules which may be difficult to construct otherwise.

\vspace{.2in}

\section{The algebra $(S(\g) \otimes \twedge  (\p))^{K}$}
\label{degrees}

We use the same strategy as in \cite{Pr1}: we study the algebra $(S(\g) \otimes \twedge  (\p))^{K}$ since the algebras $S(\g) \otimes \twedge  (\p)$ and $U(\g) \otimes C(\p)$ are isomorphic as $K$--modules, and the algebra $S(\g) \otimes \twedge  (\p)$ is easier to study.

\bigskip

Let $G$ be the Lie group 
\[
SO(4,1) = \{ g \in GL(5, \mathbb{R}) \, | \, g^{\tau} \gamma g = \gamma \},
\]
where 
\[
\gamma = \begin{pmatrix}1 & 0 & 0 & 0 & 0\cr 0 & 1 & 0 & 0 & 0\cr 0 & 0 & 1 & 0 & 0\cr 0 & 0& 0 & 1 & 0\cr 0 & 0 & 0 & 0 & -1\end{pmatrix}.
\] 
The (real) Lie algebra of $G$ is 
\[
\g_0 = \mathfrak{so}(4,1)= \{ x \in \mathfrak{gl}(5, \mathbb{R}) \, | \, x^{\tau} = - \gamma x \gamma, \quad \text{tr} X = 0 \}.
\] 
We use the following basis for the complexification $\g$ of $\frg_0$:
\begin{gather*}
H_1 = i e_{12} - i e_{21},\quad H_2  = i e_{34} - i e_{43}, \\ 
E_1  = \frac{1}{2}(e_{13} - e_{24} - ie_{23} - ie_{14} - e_{31} + e_{42} + ie_{32} + ie_{41}), \\
E_2  = \frac{1}{2}(e_{13} + e_{24} - ie_{23} + ie_{14} - e_{31} - e_{42} + ie_{32} - ie_{41}), \\
F_1  = \frac{-1}{2}(e_{13} - e_{24} + ie_{23} + ie_{14} - e_{31} + e_{42} - ie_{32} - ie_{41}), \\
F_2  = \frac{-1}{2}(e_{13} + e_{24} + ie_{23} - ie_{14} - e_{31} - e_{42} - ie_{32} + ie_{41}), \\
E_3  = e_{15} - ie_{25} + e_{51} - ie_{52}, \\
E_4  = e_{35} - ie_{45} + e_{53} - ie_{54}, \\
F_3  = e_{15} + ie_{25} + e_{51} + ie_{52}, \\
F_4  = e_{35} + ie_{45} + e_{53} + ie_{54}.
\end{gather*}
where $e_{ij}$ denotes the matrix with the $ij$ entry equal to $1$ and all other entries equal to $0$. 
The commutation relations are given by the following table:
\begin{table}[ht]
\caption{commutator table}
\label{tab:commtable}
\begin{tabular}{|r|| r| r| r| r| r| r| r| r| r|}
  \hline
        & $H_2$ & $E_1$ & $E_2$ & $F_1$         & $F_2$       & $E_3$ & $E_4$  & $F_3$  & $F_4$ \\ \hline \hline
  $H_1$ & $0$   & $E_1$ & $E2$  & $-F_1$        & $-F_2$      & $E_3$ & $0$    & $-F_3$ & $0$   \\ \hline
  $H_2$ &       & $E_1$ & $-E2$ & $-F_1$        & $F_2$       & $0$   & $E_4$  & $0$    & $-F_4$   \\ \hline
  $E_1$ &       &       & $0$   & $H_1 + H_2$   & $0$         & $0$   & $0$    & $-E_4$ & $E_3$   \\ \hline
  $E_2$ &       &       &       & $0$           & $H_1 - H_2$ & $0$   & $E_3$  & $-F_4$ & $0$   \\ \hline
  $F_1$ &       &       &       &               & $0$         & $F_4$ & $-F_3$ & $0$    & $0$   \\ \hline
  $F_2$ &       &       &       &               &             & $E_4$ & $0$    & $0$    & $-F_3$   \\ \hline
  $E_3$ &       &       &       &               &             &       & $2E_1$ & $2H_1$ & $2E_2$   \\ \hline
  $E_4$ &       &       &       &               &             &       &        & $2F_2$ & $2H_2$   \\ \hline
  $F_3$ &       &       &       &               &             &       &        &        & $-2F_1$   \\ \hline
\end{tabular}
\end{table}

Let $\g = \ka \oplus \p$ be the Cartan decomposition of $\g$ corresponding to the Cartan involution $\theta (X) = - X^{\tau}$. Then
\[
\ka = \cspan  \{ H_1, H_2, E_1, E_2, F_1, F_2 \} \cong \mathfrak{so}(4, \mathbb{C}),\text{ and } \p = \cspan  \{ E_3, E_4, F_3, F_4 \}.
\]
We have $\ka = \ka_1 \oplus \ka_2$, where 
\[
\ka_1 = \cspan\{ H_1 + H_2, E_1, F_1 \} \simeq \mathfrak{sl}(2, \mathbb{C}), \quad \ka_2 = \cspan\{ H_1 - H_2, E_2, F_2\} \simeq \mathfrak{sl}(2, \mathbb{C})
\]
with $H_1 + H_2$, $E_1$ and $F_1$ (resp. $H_1 - H_2$, $E_2$ and $F_2$) corresponding to the standard basis of $\frsl(2,\mathbb{C})$. Algebras $\ka_1$ and $\ka_2$ mutually commute. 
We set
\begin{align*}
a_1 & = (H_1 + H_2)^2 + 4 E_1F_1 \in S(\ka_1) \subset S(\g) \\
a_2 & = (H_1 - H_2)^2 + 4 E_2F_2 \in S(\ka_2) \subset S(\g).
\end{align*}
Note that $a_1$ (respectively $a_2$) symmetrizes to a multiple of the Casimir element of $U(\ka_1)$ (respectively $U(\ka_2)$), and that they are both $K$--invariant.

From \cite[Lemma~2.1]{Pr1} it follows that 
\begin{equation}\label{S_k1}
S(\ka_1)^{\ka_1} = \mathbb{C}[a_1] \qquad\text{and}\qquad S(\ka_2)^{\ka_2} = \mathbb{C}[a_2].
\end{equation}
and also 
\begin{equation}\label{S_k2}
S(\ka_1) = S(\ka_1)^{\ka_1} \otimes \mathcal{H}_{\ka_1} \qquad\text{and}\qquad S(\ka_2) = S(\ka_2)^{\ka_2} \otimes \mathcal{H}_{\ka_2},
\end{equation}
with the spaces of harmonics $\mathcal{H}_{\ka_1}$ and $\mathcal{H}_{\ka_2}$ decomposing as 
\[
\mathcal{H}_{\ka_1} = \bigoplus_{n \in \mathbb{Z}_{+}} V_{(n,n)}\qquad \text{and}\qquad \mathcal{H}_{\ka_2} = \bigoplus_{n \in \mathbb{Z}_{+}} V_{(n,-n)}.
\] 
Here $V_{(n,n)}$ denotes the $\mathfrak{so}(4, \mathbb{C})$--module with the highest weight vector $E_{1}^{n}$, on which $H_1$ and $H_2$ both act by $n$. Similarly, 
$V_{(n,-n)}$ denotes the $\mathfrak{so}(4, \mathbb{C})$--module with the highest weight vector $E_{2}^{n}$, on which $H_1$ acts by $n$ and $H_2$ by $-n$.

Similarly as in the proof of \cite[Lemma~2.3.]{Pr1}, we have 
\[
S^{n}(\p) = V_{(n,0)} \oplus (E_3 F_3 + E_4 F_4) S^{n-2}(\p),
\]
where $V_{(n,0)}$ is the $\mathfrak{so}(4, \mathbb{C})$--module with the highest weight vector $E_{3}^{n}$. Let us denote 
\[
b = E_3 F_3 + E_4 F_4.
\]
One can easily check that  $b$ is $K$--invariant. Now we have
\begin{equation}\label{S_p}
S(\p) = S(\p)^{K} \otimes \mathcal{H}_{\p} = \mathbb{C}[b] \otimes \mathcal{H}_{\p},
\end{equation}
where 
\[
\mathcal{H}_{\p} = \bigoplus_{n \in \mathbb{Z}_{+}} V_{(n,0)}.
\]
From (\ref{S_k1}), (\ref{S_k2}) and (\ref{S_p}) it follows that
\[
S(\g) \otimes \twedge(\p) = \mathbb{C}[a_1, a_2, b] \otimes \mathcal{H}_{\ka_1} \otimes \mathcal{H}_{\ka_2} \otimes \mathcal{H}_{\p} \otimes \twedge(\p).
\]

The $\frk$-submodules of $\twedge(\frp)$ are: 

\begin{align*}
& \twedge^0(\frp) \cong  \twedge^4(\frp) \cong  V_{(0,0)} \\
& \twedge^1(\frp)\cong\twedge^3(\frp) \cong   V_{(1,0)} \\
& \cspan   \{ E_3 \wedge E_4, E_3 \wedge F_3 + E_4 \wedge F_4, F_3 \wedge F_4 \} \cong   V_{(1,1)} \\
& \cspan   \{ E_3 \wedge F_4, E_3 \wedge F_3 - E_4 \wedge F_4, E_4 \wedge F_3 \} \cong   V_{(1, -1)}.
\end{align*}
It follows that $\twedge(\frp)$ decomposes under $\frk$ as 
\begin{align*}
\twedge  (\p) & =  \overbrace{V_{(0, 0)}}^{0} \\
& \oplus  \overbrace{V_{(1,0)}}^{1} \\
& \oplus  \overbrace{V_{(1,1)}}^{2} \oplus \overbrace{V_{(1,-1)}}^{2} \\
& \oplus  \overbrace{V_{(1,0)}}^{3} \\
& \oplus  \overbrace{V_{(0,0)}}^{4},
\end{align*}
where each of the numbers over braces denotes the degree in which the corresponding $\frk$-module is appearing.

If $V$ is a $\ka_1$--module and $W$ is a $\ka_2$--module, then we  denote by $V \boxtimes W$ the external tensor product of $V$ and $W$. As a vector space, $V \boxtimes W$ is equal to the direct sum of $V$ and $W$, and the action of $\ka=\ka_1\oplus\ka_2$--module is defined so that $\ka_1$ acts on $V$ and $\ka_2$ acts on $W$. Then we have 
\[
V_{(n,m)} \simeq V_{n+m} \boxtimes V_{n-m} \quad\text{ as $\ka$--modules,}
\]
where $V_k$ denotes the $\frsl(2,\mathbb{C})$--module with highest weight $k$. It follows that 
\begin{align*}
\mathcal{H}_{\ka_1} \otimes \mathcal{H}_{\ka_2} 
& = \left ( \bigoplus_{n \in \mathbb{Z}_{+}} V_{(n,n)} \right ) \otimes \left ( \bigoplus_{m \in \mathbb{Z}_{+}} V_{(m,-m)} \right ) \\
& \simeq \left ( \bigoplus_{n \in \mathbb{Z}_{+}} V_{2n} \boxtimes V_{0}  \right ) \otimes \left ( \bigoplus_{m \in \mathbb{Z}_{+}} V_{0} \boxtimes V_{2m} \right ) \\
& \simeq \bigoplus_{n,m \in \mathbb{Z}_{+}}\left (  V_{2n} \otimes V_0 \right ) \boxtimes \left (  V_{0} \otimes V_{2m} \right ) \\
& \simeq \bigoplus_{n,m \in \mathbb{Z}_{+}} V_{2n} \boxtimes V_{2m} \\
& \simeq \bigoplus_{n,m \in \mathbb{Z}_{+}} V_{(n+m, n-m)}.
\end{align*}
We have seen that $\mathcal{H}_{\p} = \bigoplus_{k \in \mathbb{Z}_{+}} V_{(k,0)}$. On the other hand, 
\begin{align}\label{product}
& V_{(n+m, n-m)} \otimes V_{(k, 0)}
 = \left ( V_{2n} \boxtimes V_{2m} \right ) \otimes \left ( V_k \boxtimes V_k \right ) \\
 &\qquad  = \left ( V_{2n} \otimes V_{k} \right ) \boxtimes \left ( V_{2m} \otimes V_k \right ) \notag \\
&\qquad = \left ( V_{2n+k} \oplus V_{2n+k-2} \oplus \cdots \oplus V_{|2n - k|} \right ) \boxtimes \left ( V_{2m+k} \oplus V_{2m+k-2} \oplus \cdots \oplus V_{|2m - k|} \right ) \notag \\
& \qquad = \left ( V_{2n + k} \boxtimes V_{2m + k} \right ) \oplus \left ( V_{2n + k} \boxtimes V_{2m + k - 2} \right ) \oplus \cdots \oplus \left ( V_{|2n - k|} \boxtimes V_{|2m - k|} \right ). \notag
\end{align}
Furthermore, we have
\begin{align*}
\twedge (\p) & = \left ( \overbrace{V_0 \boxtimes V_0}^{0} \right ) \\
& \oplus \left ( \overbrace{V_1 \boxtimes V_1}^{1} \right ) \\
& \oplus \left ( \overbrace{V_2 \boxtimes V_0}^{2} \right ) \oplus \left ( \overbrace{V_0 \boxtimes V_2}^{2} \right ) \\
& \oplus \left ( \overbrace{V_1 \boxtimes V_1}^{3} \right ) \\
& \oplus \left ( \overbrace{V_0 \boxtimes V_0}^{4} \right ).
\end{align*}
where as before, the numbers over braces denote the degrees in $\twedge(\p)$.

In order to get an invariant, we can tensor $V_0 \boxtimes V_0$ only with $V_0 \boxtimes V_0$. In (\ref{product}), we have $V_0 \boxtimes V_0$ only if $|2n - k| = |2m - k| = 0$, that is, $2n = 2m = k$. In that case, we have one invariant in degree $4n$ and one invariant in degree $4n + 4$, where $n \geq 0$. 

Furthermore, $V_1 \boxtimes V_1$ can be tensored only with $V_1 \boxtimes V_1$, and in (\ref{product}) we have $V_1 \boxtimes V_1$ only for $|2n - k| = |2m - k| = 1$. We have several cases:
\begin{itemize}
\item If $n = m, k = 2n + 1$, we have one invariant in  degree $4n + 2$ and one invariant in degree $4n + 4$, where $n\geq 0$. 
\item If $n = m, k = 2n - 1$, we have one invariant in  degree $4n$ and one invariant in degree $4n + 2$, where $n\geq 1$.
\item If $m = n + 1, k = 2n + 1$, we have one invariant in degree $4n + 3$ and one invariant in degree $4n + 5$, where $n\geq 0$. 
\item If $m = n - 1, k = 2n - 1$, we have one invariant in degree $4n - 1$ and one invariant in degree $4n + 1$, where $n\geq 1$.
\end{itemize}

The exterior product $V_2 \boxtimes V_0$ can be tensored only with $V_2 \boxtimes V_0$. In (\ref{product}), we have $V_2 \boxtimes V_0$ only for $|2n - k| = 2, |2m - k| = 0$ or $|2n - k| + 2 = 2, |2m - k| = 0$. The cases are:
\begin{itemize}
\item If $k = 2n$, we have one invariant in degree $4n + 2$, where $n \geq 1$. 
\item If $2 n - k = 2$, we have one invariant in degree $4n - 1$, where $n \geq 1$. 
\item If $2n - k = -2$, we have one invariant in degree $4n + 5$, where $n\geq 0$. 
\end{itemize}

Finally, $V_0 \boxtimes V_2$ can be tensored only with $V_0 \boxtimes V_2$, and in (\ref{product}) we have $V_0 \boxtimes V_2$ only for $|2n - k| = 0, |2m - k| = 2$ or $|2n - k| = 0, |2m - k|  + 2 = 2$. The cases are:
\begin{itemize}
\item If $m = n$, we have one invariant in degree $4n + 2$, where $n \geq 1$. 
\item If $2 m - k = 2$, we have one invariant in degree $4n + 3$, where $n \geq 0$. 
\item If $2m - k = -2$, we have one invariant in degree $4n + 1$, where $n\geq 1$. 
\end{itemize}

From all the above we conclude that we have the following table for the $K$--invariants in $\mathcal{H}_{\ka_1} \otimes \mathcal{H}_{\ka_2} \otimes \mathcal{H}_{\p} \otimes \twedge(\p)$: 

\begin{center}
\label{table:degree}
\begin{tabular}{|c|c|}
  \hline
  degree & number of linearly independent invariants \\ \hline
  $0$ & $1$ \\
  $1$ & $0$ \\
  $2$ & $1$ \\
  $4k, \quad k \geq 1$ & $4$ \\
  $4k+1, \quad k \geq 1$ & $4$ \\
  $4k+2, \quad k \geq 1$ & $4$ \\
  $4k+3, \quad k \geq 1$ & $4$ \\
  \hline
\end{tabular}
\end{center}
We have proved:
\begin{prop}
\label{number inv}
The algebra of $K$--invariants in $S(\frg)\otimes\twedge(\frp)$ can be written as
\[
(S(\frg)\otimes\twedge(\frp))^K=\bbC[a_1,a_2,b]\otimes (\mathcal{H}_{\ka_1} \otimes \mathcal{H}_{\ka_2} \otimes \mathcal{H}_{\p} \otimes \twedge(\p))^K.
\]
The number of invariants in $(\mathcal{H}_{\ka_1} \otimes \mathcal{H}_{\ka_2} \otimes \mathcal{H}_{\p} \otimes \twedge(\p))^K$ in each degree is given by the above table.
\end{prop}

The elements $a_1,a_2$ and $b$ of $S(\frg)$ can be viewed as elements of $S(\g) \otimes \twedge(\p)$ by the identification
$S(\frg)=S(\frg)\otimes 1$. We define further elements, which are all easily checked to be $K$-invariant:
\begin{align*}
 D  & =  E_3 \otimes F_3 + E_4 \otimes F_4 + F_3 \otimes E_3 + F_4 \otimes E_4, \quad (\text{ the Dirac operator})\\
 c  & = (2 (E_1 E_2 F_{3}^{2} - E_1 F_2 F_{4}^{2} + F_1 F_2 E_{3}^{2} - F_1 E_2 E_{4}^{2}) \\
 & - 2(H_1 - H_2)(E_1 F_3 F_4 + F_1 E_3 E_4) - 2(H_1 + H_2)(F_2 E_3 F_4 + E_2 F_3 E_4) \\
 & - (H_1 - H_2)(H_1 + H_2)(E_3 F_3 - E_4 F_4)) \otimes 1, \\
 d  & = 2 E_1 \otimes F_3 \wedge F_4 - (H_1 + H_2) \otimes (E_3 \wedge F_3 + E_4 \wedge F_4) - 2 F_1 \otimes E_3 \wedge E_4, \\
 e  & = 2 E_2 \otimes E_4 \wedge F_3 + (H_1 - H_2) \otimes (E_3 \wedge F_3 - E_4 \wedge F_4) + 2 F_2 \otimes E_3 \wedge F_4, \\
 f  & = 2 (E_1 F_3 \otimes F_4 - E_1 F_4 \otimes F_3) \\ 
 & -(H_1 + H_2) (E_3 \otimes F_3 + E_4 \otimes F_4 - F_3 \otimes E_3 - F_4 \otimes E_4) \\
 & - 2 (F_1 E_3 \otimes E_4 - F_1 E_4 \otimes E_3), \\
 g  & = -2 (E_2 F_3 \otimes E_4 - E_2 E_4 \otimes F_3)\\ 
 & + (H_1 - H_2)(E_3 \otimes F_3 - E_4 \otimes F_4 - F_3 \otimes E_3 + F_4 \otimes E_4) \\
 & + 2(F_2 E_3 \otimes F_4 - F_2 F_4 \otimes E_3) \\
 h  & = 2 (E_1 E_2 F_3 \otimes F_3 - E_1 F_2 F_4 \otimes F_4 + F_1 F_2 E_3 \otimes E_3 - F_1 E_2 E_4 \otimes E_4) \\
  & - (H_1 - H_2)(E_1 F_3 \otimes F_4 + E_1 F_4 \otimes F_3 + F_1 E_3 \otimes E_4 + F_1 E_4 \otimes E_3) \\
  & - (H_1 + H_2)(F_2 E_3 \otimes F_4 + F_2 F_4 \otimes E_3 + E_2 F_3 \otimes E_4 + E_2 E_4 \otimes F_3) \\
  & - \frac{1}{2}(H_1 - H_2)(H_1 + H_2)(E_3 \otimes F_3 + F_3 \otimes E_3 - E_4 \otimes F_4 - F_4 \otimes E_4), \\
 i  & = 1 \otimes E_3 \wedge E_4 \wedge F_3 \wedge F_4, \\
 j  & = E_3 \otimes E_4 \wedge F_3 \wedge F_4 - E_4 \otimes E_3 \wedge F_3 \wedge F_4\\ & + F_3 \otimes E_3 \wedge E_4 \wedge F_4 - F_4 \otimes E_3 \wedge E_4 \wedge F_3.
\end{align*}
\begin{prop}\label{basis}
Let $S$ and $T$ be the following subsets of $(S(\g) \otimes \twedge(\p))^{K}$:
\begin{align*}
& S = \{ a_{1}^{n_1} a_{2}^{n_2} b^{n_3} c^{n_4} \, | \, n_1, n_2, n_3, n_4 \in \mathbb{N}_{0} \} \\
& T = \{ 1, D, d, e, f, g, h, i, j, Dd, De, Df, Dg, fg, Dh, dg \}.
\end{align*}
Then the set $S \cdot T$ of products of elements of $S$ and $T$ in the algebra $S(\g) \otimes \twedge(\p)$ is a
basis for $(S(\g) \otimes \twedge(\p))^{K}$.
\end{prop}

\begin{proof}
We first prove the linear independence of the set $S \cdot T$. Notice that it is enough to prove the linear independence of following sets:
\begin{enumerate}[a)]
\item $S$,
\item $S \cdot \{ D \} \cup S \cdot \{ f \} \cup S \cdot \{ g \} \cup S \cdot \{ h \}$,
\item $S \cdot \{ d \} \cup S \cdot \{ Df \} \cup S \cdot \{ e \} \cup S \cdot \{ Dg \} \cup S \cdot \{ Dh \} \cup S \cdot \{ fg \}$,
\item $S \cdot \{ j \} \cup S \cdot \{ Dd \} \cup S \cdot \{ De \} \cup S \cdot \{ dg \}$.
\end{enumerate}
The rest of the independence then follows by considering just the second factors in the tensor products.

\begin{enumerate}[a)]
\item Let $\sum_{i \in \mathcal{I}} \lambda_{i} \cdot a_{1}^{n_{1,i}} a_{2}^{n_{2,i}} b^{n_{3, i}} c^{n_{4, i}} = 0$. In the expansion of the summand $a_{1}^{n_{1,i}} a_{2}^{n_{2,i}} b^{n_{3, i}} c^{n_{4, i}}$ we consider the terms without $H_1 - H_2$, $E_1$, $E_3$, $F_3$. There is only one such term and it is 
\[
(H_1 + H_2)^{2 n_{1, i}} (4 E_2 F_2)^{n_{2, i}}(E_4 F_4)^{n_{3, i}}(-2 F_1 E_2 E_{4}^{2})^{n_{4, i}}. 
\]
We have 
\[
\sum_{i \in \mathcal{I}} \lambda_{i} (H_1 + H_2)^{2 n_{1, i}} (4 E_2 F_2)^{n_{2, i}}(E_4 F_4)^{n_{3, i}}(-2 F_1 E_2 E_{4}^{2})^{n_{4, i}} = 0,
\] 
that is, 
\[
\sum_{i \in \mathcal{I}} \lambda_{i} 4^{n_{2, i}} \cdot (-2)^{n_{4, i}} (H_1 + H_2)^{2 n_{1, i}} E_{2}^{n_{2, i} + n_{4, i}} E_{4}^{n_{3, i} + 2 n_{4, i}} F_{1}^{n_{4, i}} F_{2}^{n_{2, i}} F_{4}^{n_{3, i}} = 0.
\]
It follows from the Poincar\'{e}-Birkhoff-Witt theorem that $\lambda_i = 0$ for all $j \in \{ 1, \cdots, l \}$.

\item We consider summands of the form $\cdot \otimes F_3$. It is enough to show the linear independence of the set
\begin{align*}
& \{a_{1}^{n_1} a_{2}^{n_2} b^{n_3} c^{n_4} E_3 \, | \, n_1, n_2, n_3, n_4 \in \mathbb{N}_{0} \}\ \cup \\
& \{a_{1}^{n_1} a_{2}^{n_2} b^{n_3} c^{n_4} \left (- 2 E_1 F_4 - (H_1 + H_2) E_3 \right ) \, | \, n_1, n_2, n_3, n_4 \in \mathbb{N}_{0} \} \ \cup \\
& \{a_{1}^{n_1} a_{2}^{n_2} b^{n_3} c^{n_4} \left ( 2 E_2 E_4 + (H_1 - H_2) E_3 \right ) \, | \, n_1, n_2, n_3, n_4 \in \mathbb{N}_{0} \} \ \cup \\
& \{a_{1}^{n_1} a_{2}^{n_2} b^{n_3} c^{n_4} ( 2 E_1 E_2 F_3 - (H_1 - H_2) E_1 F_4 - (H_1 + H_2) E_2 E_4 - \\
&\qquad\qquad \frac{1}{2}(H_1 - H_2)(H_1 + H_2) E_3 ) \, | \, 
 n_1, n_2, n_3, n_4 \in \mathbb{N}_{0} \}.
\end{align*}

As in case $a)$, we consider the terms without $H_1 - H_2$, $E_1$, $E_3$, $F_3$. We have
\begin{align*}
& \sum_{k \in \mathcal{K}} \lambda_k (H_1 + H_2)^{2 n_{1, k}} (4 E_2 F_2)^{n_{2, k}}(E_4 F_4)^{n_{3, k}}(-2 F_1 E_2 E_{4}^{2})^{n_{4, k}} \cdot (2 E_2 E_4) + \\
& \sum_{l \in \mathcal{L}} \lambda_l (H_1 + H_2)^{2 n_{1, l}} (4 E_2 F_2)^{n_{2, l}}(E_4 F_4)^{n_{3, l}}(-2 F_1 E_2 E_{4}^{2})^{n_{4, l}} \left ( - (H_1 + H_2) E_2 E_4 \right ) = 0. 
\end{align*}
It follows that
\[
\sum_{k \in \mathcal{K}} \lambda_k (H_1 + H_2)^{2 n_{1, k}} (4 E_2 F_2)^{n_{2, k}}(E_4 F_4)^{n_{3, k}}(-2 F_1 E_2 E_{4}^{2})^{n_{4, k}} = 0
\]
and
\[
\sum_{l \in \mathcal{L}} \lambda_l (H_1 + H_2)^{2 n_{1, l} + 1} (4 E_2 F_2)^{n_{2, l}}(E_4 F_4)^{n_{3, l}}(-2 F_1 E_2 E_{4}^{2})^{n_{4, l}} = 0. 
\]

By the same arguments as in case $a)$, we get $\lambda_k = 0$ for all $k \in \mathcal{K}$ and $\lambda_l = 0$ for all $l \in \mathcal{L}$. Then we have 
\[
\sum_{i \in \mathcal{I}}\lambda_i a_{1}^{n_{1, i}} a_{2}^{n_{2, i}} b^{n_{3, i}} c^{n_{4, i}} E_3 + \sum_{j \in \mathcal{J}} \lambda_j a_{1}^{n_{1, j}} a_{2}^{n_{2, j}} b^{n_{3, j}} c^{n_{4, j}} \left ( - 2 E_1 F_4 - (H_1 + H_2) E_3 \right ) = 0.
\] 
Now we consider the terms without $H_1 + H_2$, $E_2$, $E_4$, $F_4$. We have
\[
\sum_{i \in \mathcal{I}} \lambda_i  (4 E_1 F_1)^{n_{1, i}} (H_1 - H_2)^{2 n_{2, i}} (E_3 F_3)^{n_{3, i}} (2 F_1 F_2 E_{3}^{2})^{n_{4, i}} E_3 = 0,
\]
which implies
\[
\sum_{i \in \mathcal{I}} \lambda_i  \cdot 4^{n_{1, i}} \cdot 2^{n_{4, i}} (H_1 - H_2)^{2 n_{2, i}} E_{1}^{n_{1, i}} E_{3}^{n_{3, i} + 2 n_{4, i}} F_{1}^{n_{1, i} + n_{4, i}} F_{2}^{n_{4, i}} F_{3}^{n_{3, i}} = 0.
\]
Hence, $\lambda_i = 0$ for all $i \in \mathcal{I}$.
So we have
\[
\sum_{j \in \mathcal{J}} \lambda_j a_{1}^{n_{1, j}} a_{2}^{n_{2, j}} b^{n_{3, j}} c^{n_{4, j}} = 0.
\]
Now $a)$ implies that $\lambda_j = 0$ for all $j \in \mathcal{J}$.

\item In this case we consider summands of the form $\cdot \otimes E_3 \wedge F_3$. It is enough to prove the linear independence of the set
\begin{align*}
 S & \cdot (-(H_1 + H_2)) \quad \cup \\
 S & \cdot (-2 E_1 F_4 F_3 - (H_1 + H_2) E_3 F_3 - (H_1 + H_2) E_3 F_3 - 2 F_1 E_3 E_4) \quad \cup \\
 S & \cdot (H_1 - H_2) \quad \cup \\
 S & \cdot (F_3 (2 E_2 E_4 + (H_1 - H_2) E_3) + E_3 ((H_1 - H_2) F_3 + 2 F_2 F_4)) \quad \cup \\
 S & \cdot ( F_3(2 E_1 E_2 F_3 - (H_1 - H_2)E_1 F_4 - (H_1 + H_2)E_2 E_4 
   - \frac{1}{2}(H_1 - H_2)(H_1 + H_2)E_3 ) \\
   & + E_3 (- 2 F_1 F_2 E_3 + (H_1 - H_2)F_1 E_4 + (H_1 + H_2)F_2 F_4 
     + \frac{1}{2}(H_1 - H_2)(H_1 + H_2)F_3 ) \quad \cup \\
 S & \cdot ((H_1 + H_2)F_3 + 2 F_1 E_4)(2 E_2 E_4 + (H_1 - H_2)E_3) \\
   & + (2 E_1 F_4 + (H_1 + H_2)E_3)(- (H_1 - H_2)F_3 - 2 F_2 F_4)).
\end{align*}
Let
\begin{align*}
& \sum_{i \in \mathcal{I}} \lambda_i a_{1}^{n_{1, i}} a_{2}^{n_{2, i}} b^{n_{3, i}} c^{n_{4, i}}(-(H_1 + H_2)) + \\
& \sum_{j \in \mathcal{J}} \lambda_j a_{1}^{n_{1, j}} a_{2}^{n_{2, j}} b^{n_{3, j}} c^{n_{4, j}} (-2 E_1 F_4 F_3 - (H_1 + H_2) E_3 F_3 - (H_1 + H_2) E_3 F_3 - 2 F_1 E_3 E_4) + \\
& \sum_{k \in \mathcal{K}} \lambda_k a_{1}^{n_{1, k}} a_{2}^{n_{2, k}} b^{n_{3, k}} c^{n_{4, k}} (H_1 - H_2) +  \\
& \sum_{l \in \mathcal{L}} \lambda_l a_{1}^{n_{1, l}} a_{2}^{n_{2, l}} b^{n_{3, l}} c^{n_{4, l}} (F_3 (2 E_2 E_4 + (H_1 - H_2) E_3) + E_3 ((H_1 - H_2) F_3 + 2 F_2 F_4)) + \\
& \sum_{m \in \mathcal{M}} \lambda_m a_{1}^{n_{1, m}} a_{2}^{n_{2, m}} b^{n_{3, m}} c^{n_{4, m}} (F_3(2 E_1 E_2 F_3 - (H_1 - H_2)E_1 F_4 - (H_1 + H_2)E_2 E_4 \\
& - \frac{1}{2}(H_1 - H_2)(H_1 + H_2)E_3) + E_3(- 2 F_1 F_2 E_3 + (H_1 - H_2)F_1 E_4 + (H_1 + H_2)F_2 F_4 \\
& + \frac{1}{2}(H_1 - H_2)(H_1 + H_2)F_3) + \\
& \sum_{n \in \mathcal{N}} \lambda_n a_{1}^{n_{1, n}} a_{2}^{n_{2, n}} b^{n_{3, n}} c^{n_{4, n}}((H_1 + H_2)F_3 + 2 F_1 E_4)(2 E_2 E_4 + (H_1 - H_2)E_3) \\
& + (2 E_1 F_4 + (H_1 + H_2)E_3)(- (H_1 - H_2)F_3 - 2 F_2 F_4)) = 0.
\end{align*}
We first consider the summands without $H_1 - H_2, E_1, E_3$ and $F_3$ and we get, similarly as in case $a)$, that $\lambda_i = 0$ for all $i \in \mathcal{I}$ and $\lambda_n = 0$ for all $n \in \mathcal{N}$. Now we consider summands without $H_1 + H_2, E_2,  E_3$ and $ F_3$, and we get that $\lambda_k = 0$ for all $k \in \mathcal{K}$.
Then we consider summands without $H_1 + H_2, E_2, E_4$ and $F_4$, and we get that $\lambda_l = 0$ for all $l \in \mathcal{L}$ and $\lambda_m = 0$ for all $m \in \mathcal{M}$. From this we conclude that also $\lambda_j = 0$ for all $j \in \mathcal{J}$.

\item This time we consider summands of the form $\cdot \otimes E_4 \wedge F_3 \wedge F_4$. It is enough to show linear independence of the set

\begin{align*}
& \{a_{1}^{n_1} a_{2}^{n_2} b^{n_3} c^{n_4} E_3 \, | \, n_1, n_2, n_3, n_4 \in \mathbb{N}_{0} \} \cup \\
& \{a_{1}^{n_1} a_{2}^{n_2} b^{n_3} c^{n_4} \left ((H_1 + H_2) E_3  + 2 E_1 F_4 \right ) \, | \, n_1, n_2, n_3, n_4 \in \mathbb{N}_{0} \} \cup \\
& \{a_{1}^{n_1} a_{2}^{n_2} b^{n_3} c^{n_4} \left ( (H_1 - H_2) E_3 - 2 E_2 E_4 \right ) \, | \, n_1, n_2, n_3, n_4 \in \mathbb{N}_{0} \} \cup \\
& \{a_{1}^{n_1} a_{2}^{n_2} b^{n_3} c^{n_4} ( -4 E_1 E_2 F_3 + 2 (H_1 - H_2) E_1 F_4 + 2 (H_1 + H_2) E_2 E_4 +\\
&\qquad\qquad (H_1 + H_2)(H_1 - H_2) E_3 ) \, | \, 
 n_1, n_2, n_3, n_4 \in \mathbb{N}_{0} \}.
\end{align*}
We consider the summands without $H_1 - H_2$, $E_1$, $E_4$, $F_4$ and we get $\lambda_i = 0$ for all $i \in \mathcal{I}$ and $\lambda_j = 0$ for all $j \in \mathcal{J}$. Then we consider the summands without $H_1 + H_2$, $E_2$, $E_3$, $F_3$ and we get $\lambda_l = 0$ for all $l \in \mathcal{L}$ and $\lambda_k = 0$ for all $k \in \mathcal{K}$.
\end{enumerate}
This finishes the proof of linear independence of the set $S\cdot T$. To prove that $S\cdot T$ is also a spanning set, we consider the degrees of the elements of the set $\{c^{n_4}\}\cdot T$.
The degrees of these elements are respectively
\begin{align*}
& 4 n_4, 4 n_4 + 2, 4 n_4 + 3, 4 n_4 + 3, 4 n_4 + 3, 4 n_4 + 3, 4 n_4 + 4, 4 n_4 + 4, 4 n_4 + 4 \\
& 4 n_4 + 5, 4 n_4 + 5, 4 n_4 + 5, 4 n_4 + 5, 4 n_4 + 6, 4 n_4 + 6, 4 n_4 + 6.
\end{align*}
Considering the number of invariants in each degree of the set $\{c^{n_4}\,\big|\, n_4\in\bbZ_+\}\cdot T$, we get the following table:

\bigskip

\begin{center}
\begin{tabular}{|c|c|}
  \hline
  degree & number of invariants \\ \hline
  $0$ & $1$ \\
  $1$ & $0$ \\
  $2$ & $1$ \\
  $4k, \quad k \geq 1$ & $4$ \\
  $4k+1, \quad k \geq 1$ & $4$ \\
  $4k + 2, \quad k \geq 1$ & $4$ \\
  $4k + 3, \quad k \geq 0$ & $4$ \\
  \hline
\end{tabular}
\end{center}
This finishes the proof, since we got the same table as in Proposition \ref{number inv}.
\end{proof}

\section{The set of generators for $(U(\g) \otimes C(\p))^{K}$}

Recall that the symmetrization map $\sigma : S(\g) \longrightarrow U(\g)$ given by
\[
\sigma(x_1 x_2 \cdots x_n)= \frac{1}{n!} \sum_{\alpha \in S_n} x_{\alpha(1)} x_{\alpha(2)} \cdots x_{\alpha(n)}, \quad n \in \mathbb{N}, x_1, \cdots, x_n \in \g
\]
is an isomorphism of $K$--modules. The Chevalley map $\tau: \twedge(\p) \longrightarrow C(\p)$ given by 
\[
\tau(v_1 \wedge \cdots \wedge v_n) = \frac{1}{n!} \sum_{\alpha \in S_n} \text{sgn}(\alpha) v_{\alpha(1)} v_{\alpha(2)}\cdots v_{\alpha(n)}
\]
is also an isomorphism of $K$--modules.

Since for $z_1, \cdots, z_n \in \g$ and $\alpha \in S_n$ we have 
\[
z_1 \cdots z_n - z_{\alpha(1)} \cdots z_{\alpha(n)} \in U_{n-1}(\g),
\] 
it follows that
\[
\sigma(z_1 \cdots z_n) = z_1 \cdots z_n \quad\text{modulo } U_{n-1}(\g).
\]
Similarly, since for $y_1, \cdots, y_k \in \p$ and $\alpha \in S_k$ we have 
\[
y_1 \cdots y_k - \text{sgn}(\alpha) y_{\alpha(1)} \cdots y_{\alpha(k)} \in C_{k-1}(\p),
\]
it follows that
\[
\tau(y_1 \wedge \cdots \wedge y_k) = y_1 \cdots y_k \quad\text{modulo }  C_{k-1}(\p).
\]
So we have 
\begin{equation}\label{universal}
\sigma(xy) = \sigma(x) \sigma(y) + U_{n + m - 1}(\g), \quad x \in U_{n}(\g), y \in U_{m}(\g),
\end{equation}
and
\begin{equation}\label{clifford}
\tau(zw) = \tau(z) \tau(w) + C_{k + l - 1}(\g), \quad z \in C_{k}(\p), w \in C_{l}(\p).
\end{equation}

Let $\rho = \sigma \otimes \tau$. Using \eqref{universal}, \eqref{clifford} and Proposition \ref{basis} one shows by induction that the following lemma holds:

\begin{lemma}
\label{ugcpk12gen}
The algebra $(U(\g) \otimes C(\p))^{K}$ is generated by elements $\rho (a_1)$, $\rho(a_2)$, $\rho(b)$, $\rho(c)$, $\rho(D)$, $\rho(d)$, $\rho(e)$, $\rho(f)$, $\rho(g)$, $\rho(h)$, $\rho(i)$ and $\rho(j)$. \qed
\end{lemma}

We define the $\ka$-Dirac operator $D_{\ka}\in U(\g)\otimes C(\p)$ mentioned in the introduction by 
\begin{align*}
D_{\ka} & = E_1 \otimes \alpha(2 F_1) + E_2 \otimes \alpha(2 F_2) + F_1 \otimes \alpha(2 E_2) + F_2 \otimes \alpha(2 E_2) \\
& + (H_1 - H_2) \otimes \alpha(H_1 - H_2) + (H_1 + H_2) \otimes \alpha(H_1 + H_2).
\end{align*}

We can now reduce the set of generators for $(U(\g) \otimes C(\p))^{K}$ given by Lemma \ref{ugcpk12gen}, since we have
\begin{align*}
\rho(b) & = -\frac{1}{2} D^2 + D_{\ka}, \\
\rho(d) & = D_{\ka}-\frac{1}{2} \Bigl (D_{\ka} \cdot \rho(i)  + \rho(i) \cdot D_{\ka} \Bigr ), \\ 
\rho(e) & = -D_{\ka}-\frac{1}{2} \Bigl (D_{\ka} \cdot \rho(i)  + \rho(i) \cdot D_{\ka} \Bigr ), \\
\rho(j) & = \frac{1}{2} \Bigl ( \rho(i) \cdot D  - D \cdot \rho(i) \Bigr ), \\
\rho(f) & = \frac{1}{2} \Bigl (\rho(d) \cdot D - D \cdot \rho(d) - 3 \rho(j) \Bigr ), \\
\rho(g) & = \frac{1}{2} \Bigl (\rho(e) \cdot D - D \cdot \rho(e) - 3 \rho(j) \Bigr ), \\
\rho(h) & = \frac{1}{4} \left ( \rho(d) \left (\rho(g) + \frac{3}{2}D \right ) + \rho(e)\left (\rho(f) - \frac{3}{2}D \right ) - \frac{3}{4}\Bigl ( \rho(f) - \rho(g) - D \Bigr ) \right ), \\
\rho(c) & = \frac{1}{4} \Biggl ( \left ( \rho(f) - \frac{3}{2}D \right ) \cdot \left ( \rho(g) + \frac{3}{2}D \right ) + \left ( \rho(g) + \frac{3}{2}D \right ) \cdot \left ( \rho(f) - \frac{3}{2}D \right ) \\
& - (\rho(a_2) \cdot \rho(d) + \rho(d) \cdot \rho(a_2)) + (\rho(a_1) \cdot \rho(e) + \rho(e) \cdot \rho(a_1)) \\
& - \frac{1}{2} \Bigl ( \rho(g) \cdot D - D \cdot \rho(g) \Bigr ) + \frac{1}{2} \Bigl (\rho(f) \cdot D - D \cdot \rho(f) \Bigr ) \\
& + (4 \cdot \rho(b) + 5) \cdot (\rho(d) - \rho(e))  + 6 \cdot \rho(b) - 8 \cdot \rho(a_1) - 8 \cdot \rho(a_2) \Biggr ).
\end{align*}

From this we conclude 

\begin{theorem}
\label{thm main} 
The algebra $(U(\g) \otimes C(\p))^{K}$ is generated by the elements $\rho(a_1)$, $\rho(a_2)$, $\rho(i) \in C(\p)^{K}$, $D$ and $D_{\ka}$.
\end{theorem}

By the same argument as in \cite[Corollary~4.5.]{Pr1}, we get
\begin{cor}
The algebra $(U(\g) \otimes C(\p))^{K}$ is a free module over $U(\g)^{K}$ of rank $16$.
\end{cor}

The last result is a special case of the more general result proved in Corollary \cite[Corollary~1]{K}.

\end{document}